\documentclass[10pt, leqno]{article}

\usepackage{amsmath,amssymb,amsthm, epsfig}

\usepackage{amsmath}

\usepackage{amssymb}

\usepackage{esint}

\usepackage{color}

\usepackage{ulem}

\usepackage{epsfig}

\usepackage[mathscr]{eucal}

\usepackage{mathptmx} 
\usepackage[scaled=0.90]{helvet} 
\usepackage{courier} 
\normalfont
\usepackage[T1]{fontenc}

\title{Cavity problems in discontinuous media}
\author{\it by \smallskip \\
Disson dos Prazeres   \quad and  \quad Eduardo V. Teixeira }

\date{}

\def \div {\mathrm{div}}
\def \dist {\mathrm{dist}}

\def \suchthat {\ \big | \ }

\def \Leb {\mathscr{L}^n}

\newtheorem{theorem}{Theorem}[section]

\newtheorem{lemma}[theorem]{Lemma}

\newtheorem{corollary}[theorem]{Corollary}

\theoremstyle{definition}

\newtheorem{definition}[theorem]{Definition}

\theoremstyle{remark}

\numberwithin{equation}{section}

\newcommand{\intav}[1]{\mathchoice {\mathop{\vrule width 6pt height 3 pt depth  -2.5pt
\kern -8pt \intop}\nolimits_{\kern -6pt#1}} {\mathop{\vrule width
5pt height 3  pt depth -2.6pt \kern -6pt \intop}\nolimits_{#1}}
{\mathop{\vrule width 5pt height 3 pt depth -2.6pt \kern -6pt
\intop}\nolimits_{#1}} {\mathop{\vrule width 5pt height 3 pt depth
-2.6pt \kern -6pt \intop}\nolimits_{#1}}}

\begin{document}
\maketitle

\begin{abstract}

We study cavitation type equations, $\text{div}(a_{ij}(X) \nabla u) \sim \delta_0(u)$, for bounded, measurable elliptic media $a_{ij}(X)$. De Giorgi-Nash-Moser theory assures that solutions are $\alpha$-H\"older continuous within its set of positivity, $\{u>0\}$, for some exponent  $\alpha$  strictly less than one. Notwithstanding, the key, main result proven in this paper provides a  sharp Lipschitz regularity estimate for such solutions along their free boundaries, $\partial \{u>0 \}$.  Such a sharp estimate implies geometric-measure constrains for the free boundary. In particular, we show that the non-coincidence $\{u>0\}$ set has uniform positive density and that the free boundary has finite $(n- \varsigma )$-Hausdorff measure, for a universal number $0<  \varsigma  \le 1$. 


\end{abstract}

\section{Introduction}

Given a Lipschitz bounded domain $\Omega \subset \mathbb{R}^n$, a bounded measurable elliptic matrix $a_{ij}(X)$, i.e. a symmetric matrix with varying coefficients satisfying the $(\lambda, \Lambda)$-ellipticity condition  
\begin{equation}\label{ellip}
	\lambda \text{Id} \leq a_{ij}(X)\leq \Lambda \text{Id},
\end{equation}
and a nonnegative boundary data $\varphi \in L^2(\partial \Omega)$, we are interested in studying  local minimizers $u$ of the discontinuous functional
\begin{equation}\label{functional AC}
	\mathscr{F} (u)=\int_{\Omega} \left \{\frac{1}{2}\langle a_{ij}(X)\nabla u,\nabla u\rangle + 							\chi_{\{u>0\}} \right \} dX \to \text{min},
\end{equation}
among all competing functions $u \in H^1_\varphi(\Omega) := \{ u \in H^1(\Omega) \suchthat \text{Trace}(u) =  \varphi \}$. 

\medskip
\par

The variational problem set in \eqref{functional AC} appears in the mathematical formulation of a great variety of models: jet flows, cavity problems, Bernoulli problems, free transmission problems, optimal designs, just to cite few. Its mathematical treatment has been extensively developed since the epic marking work of Alt and Caffarelli \cite{AC}. The program for studying minimization problems for discontinuous functionals of the form \eqref{functional AC} is nowadays well established in the literature. Existence of minimizer follows by classical considerations. Any minimum is nonnegative provided the boundary data is nonnegative. Minimizers satisfy, in the distributional sense, the Euler-Lagrange equation 
\begin{equation}
	\div (a_{ij}(X) \nabla u) = \mu,
\end{equation}
where $\mu$ is a measure supported along the free boundary.  In particular a minimum of the functional $\mathscr{F}$ is $a$-harmonic within its positive set, i.e.,
$$
	\div (a_{ij}(X) \nabla u) = 0, \quad \text{ in } \{u>0 \} \cap \Omega.	
$$
By pure energy considerations, one proves that minimizers grow linearly alway from their free boundaries. Finally, if $a_{ij}$ are, say, H\"older continuous, then the free boundary $\partial \{u>0\}$ is of class $C^{1,\alpha}$ up to a possible negligible singular set. In such a scenario, the free boundary condition
$$
	\langle a_{ij}(\xi) \nabla u(\xi), \nabla u(\xi) \rangle = \text{Const.}
$$
then holds in the classical sense along the regular part of the free boundary, in particular for all  $\xi \in \partial_{\text{red}} \{u>0\} \cap \Omega$.

\medskip
\par

A decisive, key step, though, required in the program for studying variational problems of the form \eqref{functional AC}, concerns Lipschitz estimates of minimizers. However, if no further regularity assumptions upon the coefficients $a_{ij}(X)$ is imposed, even $a$-harmonic functions, $\div (a_{ij}(X) \nabla h) = 0$, may fail to be Lipschitz continuous. That is, the universal H\"older continuity exponent granted by De Giorgi-Nash-Moser regularity theory may be  strictly less than 1, even for two-dimensional problems. Such a technical constrain makes the study of local minima to \eqref{functional AC} in discontinuous media rather difficult from a rigorous mathematical viewpoint.  

\medskip
\par

The above discussion brings us to the main goal of this present work. Even though it is hopeless to obtain gradient bounds for minimizers of functional \eqref{functional AC} in $\Omega$, we shall prove that, any minimum is universally Lipschitz continuous along its free boundary, $\partial \{ u > 0 \} \cap \Omega$, see  \cite{T2, T3} for improved estimates that hold only along (non-physical) free boundaries, see also \cite{RTU}. Such an estimate is strong enough to carry on a geometric-measure analysis near the free boundary, which in particular implies that the non-coincidence set has uniform positive density and that the free boundary has finite $(n- \varsigma )$-Hausdorff measure, for a universal number $0<  \varsigma  \le 1$. We shall establish the following result:

\begin{theorem}\label{Lip AC} Let $u$ be a nonnegative local minimum of the functional \eqref{functional AC} and $Z_0 \in \partial \{u> 0 \} \cap \Omega$ be a generic interior free boundary point. Then
$$
	C^{-1} r \le \sup\limits_{B_r(Z_0)} u \le C r,
$$
for a constant $C>0$ depending only on dimension, ellipticity constants and $\|u\|_{L^2(\Omega)}$. In particular, for another universal constant $\theta >0$,
$$
	\mathcal{L}^n(\{u>0\} \cap B_{r}(Z_0))\geq \theta r^n,
$$
	for all $0< r\ll 1$. Furthermore there is a universal constant $0< \varsigma \le 1$ such that
	$$
		\mathrm{dim}_{\mathcal{H}}(\partial\{u>0 \} )\leq n- \varsigma,
	$$
	where $dim_{\mathcal{H}}(E)$ means the Hausdorff dimension to the set $E$.
\end{theorem}
\medskip
\par

In this paper we shall develop a  more general analysis as to contemplate singular approximations of the minimization problem \eqref{functional AC}. Let $\beta \in L^\infty(\mathbb{R})$ be a bounded function supported in the unit interval $[0,1]$. For each $\epsilon >0$, we define the integral preserving, $\epsilon$-perturbed potential:
\begin{equation}\label{beta_e}
	\beta_\epsilon(t):=\frac{1}{\epsilon} \beta \left (\frac{t}{\epsilon} \right ),
\end{equation}
which is now supported in $[0,\epsilon]$. Such a sequence of potentials converges in the distributional sense to $\int \beta$ times the Dirac measure $\delta_0$. Consider further
\begin{equation}
	{B}_\epsilon(\xi)=\int_0^\xi \beta_\epsilon(t)dt \to \left (\int \beta(s) ds \right ) \cdot \chi_{\{\xi > 0 \}},
\end{equation}
in the distributional sense. We now look at local minimizers $u_\varepsilon$ to the variational problem
\begin{equation}\label{functional}
	\mathscr{F}_\epsilon(u)=\int_{B_1} \left \{\frac{1}{2}\langle a_{ij}(X)\nabla u,\nabla u\rangle + 							{B}_\epsilon(u) \right \} dX \to \text{min},
\end{equation}
among all competing functions $u \in H^1_\varphi(\Omega) := \{ u \in H^1(\Omega) \suchthat \text{Trace}(u) =  \varphi \}$. There is a large literature on such a class of singularly perturbed equations, see for instance  \cite{BCN, CLW1, CLW2, C-V, MT, RT, T}. It is well established that the functional $\mathscr{F}$ defined in \eqref{functional AC} can be recovered by letting $\epsilon$ go to zero in \eqref{functional}. For each $\epsilon$ fixed though, minimizers of the functional $\mathscr{F}_\epsilon$ is related to a number of other physical problems, such as high energy activations and the theory of flame propagation. Hence, from the applied point of view, it is more appealing to indeed study the whole family of functionals $\left ( \mathscr{F}_\epsilon \right )_{0\le \epsilon \le 1}$. We also mention that the study of minimization problem  \eqref{functional} with no continuity assumption on the coefficients is also motivated by several branch of applications, for instance in homogenization theory, composite materials, etc. 

\medskip
\par

We should also mention the connections this present work has with the theory of free phase transmission problems. This class of problems appears,  for instance, in the system of equations modeling an ice that melts submerged in a heated inhomogeneous medium. For problems modeled within an organized medium (say H\"older continuous coefficients), monotonicity formula \cite{C2} yields Lipschitz estimates for solutions. However, by physical interpretations of the model, it is natural to consider the problem within discontinuous media. Under such an adversity (monotonicity formula is no longer available), Lipschitz estimate along the free boundary has been an important open problem within that theory, see \cite{AT} for discussion. However, if we further assume in the model that the temperature of the ice remains constant, which is reasonable in very low temperatures, then free phase transmission problems fit into the mathematical formulation of this present article; and a Lipschitz estimate becomes available by our main result.

\medskip
\par

We conclude this Introduction by mentioning that the improved, sharp regularity estimate we establish in this work holds true in much more generality. Our approach to obtain Lipschitz estimate along the free boundary extends directly to degenerate discontinuous functionals of the form
$$
	\int F(X, u, \nabla u) dX \to \text{min.}, 
$$  
where 
$$
	F(X,u,\xi) \sim |\xi|^{p-2} A(X)\xi \cdot \xi + f(X)\left( u^{+}\right)^{m} + Q(X) \cdot \chi_{\{u > 0 \}},
$$
with $A(X)$ bounded, measurable elliptic matrix, $f \in L^q(\Omega)$, $q>n$, $1\le m < p$ and $Q$ is bounded away from zero and infinity, see \cite{DP, LT}. Indeed, the proof designed herein is purely nonlinear and uses solely  the Euler-Lagrange equation associated to the minimization problem \eqref{functional}. Hence,  nonvariational cavitation problems, as well as parabolic versions of such models can also be tackled by our methods. 

\section{Preliminaries}

In this Section we gather some results and tools available for the analysis of minimizers to the functional \eqref{functional} (and also to the functional \eqref{functional AC}). The results stated herein follow by methods and approaches available in the literature. We shall briefly comment on the proofs, for the readers' convenience.

\begin{theorem}[Existence of minimizers] \label{prelim-Exist}
 	For each $\epsilon>0$ fixed, there exists at least one minimizer $u_\epsilon \in H^1_\varphi(\Omega)$ to the function \eqref{functional}. Furthermore $u_\varepsilon$ satisfies 
\begin{equation}\label{e-Eq}
	\div (a_{ij}(X)\nabla u_\epsilon) = \beta_\epsilon(u_\epsilon), \quad \text{in } \Omega,
\end{equation}
in the distributional sense. Each $u_\epsilon$ is a nonnegative function, provided the boundary data $\varphi$ is nonnegative.  
\end{theorem}
\begin{proof}
Existence of minimizer as well as the Euler-Lagrange equation associated to the functional follow by classical methods in the Calculus of Variations. Non-negativity of a minimum is obtained as follows. Suppose, for the sake of contradiction, the set $\{u_\epsilon < 0\}$ were not empty. Since $\varphi \ge 0$ on $\partial \Omega$, one sees that $\partial \{u_\epsilon < 0 \} \subset \{u_\epsilon = 0\} \cap \Omega$. Since $\beta_\varepsilon$ is supported in $[0,\epsilon]$, from the equation we conclude that $u_\epsilon$ satisfies the homogeneous equation $\div (a_{ij}(X)\nabla u_\epsilon) = 0$ in $\{u_\epsilon < 0\}$. By the maximum principle we conclude $u_\epsilon \equiv 0$ in such a set, which gives a contradiction. 
 \end{proof} 

Regarding higher regularity for minimizers, it is possible to show uniform-in-$\epsilon$ $L^\infty$ bounds and also a uniform-in-$\epsilon$ $C^{0,\alpha}$ estimate, for a universal exponent $0< \alpha < 1$. 

\begin{theorem}[Uniform H\"older regularity of minimizers]\label{lim uniforme em C-alpha} 
 	Fixed a subdomain $\Omega'  \Subset \Omega$, there exists a constant $C>0$, depending on dimension, ellipticity constants, $\|\varphi\|_{L^2}$ and $\Omega' $, but independent of $\epsilon$, such that
	$$
		\|u_\epsilon\|_{L^\infty(\Omega')} + [u_\epsilon]_{C^\alpha(\Omega')}<C,
	$$ 
	where $0<\alpha<1$ is a universal number. 
 	
\end{theorem}
\begin{proof}
	The arguments to show Theorem \ref{lim uniforme em C-alpha} follow closely the ones from \cite[Theorem 3.4]{AT}, upon observing that for any ball $B_r(Y) \subset \Omega$, there too holds
	$$
		\int_{B_r(Y)} B_\epsilon (u_\epsilon) dX \le C r^n,
	$$
	for a constant $C$ independent of $\epsilon$. See also \cite[Theorem 4.4]{T0} for a result of the same flavor.
\end{proof}

As a consequence of Theorem \ref{lim uniforme em C-alpha}, up to a subsequence, $u_\epsilon$ converges locally uniformly in $\Omega$ to a nonnegative function $u_0$. By linear interpolation techniques, see for instance \cite[Theorem 5.4]{T}, one verifies that $u_0$ is a minimizer of the functional \eqref{functional AC}.

The final result we state in this section gives the sharp lower bound for the grow of $u_\epsilon$ away from $\epsilon$-level surfaces. 

\begin{theorem}[Linear Growth]\label{linear growth}
 	Let  $\Omega'  \Subset \Omega$ be a given subdomain and $X_0\in  \Omega' \cap \{u_\epsilon\geq \epsilon\}$ then 			
 	\begin{equation}\label{LG}
 		 u_\epsilon(X_0) \ge c \cdot \dist(X_0, \partial  \{u_\epsilon\geq \epsilon\}),
 	\end{equation}
	where $c$ is a constant that depends on dimension and ellipticity constants, but it is independent of $\epsilon$. 
\end{theorem}
\begin{proof}
The classical proof for linear growth is based on pure energy considerations, combined with a ``cutting hole" argument, see for instance \cite[Theorem 4.6]{T}. Hence, the same reasoning applied here yields estimate \eqref{LG}, with minor modifications.
\end{proof}

\section{Lipschitz regularity along the free boundary} \label{Sct Lip Funct}

The heart of this work lies in this Section, where we prove that uniform limits of solutions to \eqref{e-Eq} are locally Lipschitz continuous along their free boundaries. We highlight once more that our approach is purely based on the singular partial differential equation satisfied by local minimizers; therefore it can be imported to a number of other contexts, both variational and non-variational.

\begin{theorem}[Lipschitz regularity]\label{main}
	Let $u_0$ be a uniform limit point of solutions to
	$$
		\div (a_{ij}(X) \nabla u_\epsilon) = \beta_\epsilon(u_\epsilon) \quad \text{ in } \Omega
	$$ 
	and assume that $u_0(\xi)=0$. Then there exists a universal constant $C>0$, depending only on dimension, ellipticity constants,  $\dist (\xi, \partial \Omega)$ and $L^\infty$ bounds of the family such that
	$$
		|u_0(X)|\leq C|X -\xi|,
	$$
	for all point $X \in \Omega$.
\end{theorem}

Our strategy is based on a flatness improvement argument, within whom the next Lemma plays a decisive role.

\begin{lemma}\label{key lemma}
 	Fixed a ball $B_r(Y) \Subset  \Omega$ and given $\theta > 0$, there exists a $\delta > 0$, depending only on $B_r(Y)$, dimension, ellipticity constants and $L^\infty$ bounds for $u_\epsilon$, such that if 
	$$
		div (a_{ij}(X) \nabla u_\epsilon) = \delta \cdot \beta_\epsilon (u_\epsilon)
	$$
	and
	$$
		\max\{\epsilon,\inf\limits_{B_r(Y)} u_\epsilon\} \le \delta.
	$$
	Then
	$$
		\sup\limits_{B_\frac{r}{2}(Y)} u_\epsilon \le \theta.
	$$
\end{lemma}

\begin{proof}
	Let us suppose, for the sake of contradiction, that the Lemma fails to hold. There would then exist a sequence of functions $u_{\epsilon_k}$ satisfying
	$$
		\text{div} (a^k_{ij}(X) \nabla u_{\epsilon_k}) = \delta_k \beta_{\epsilon_k} (u_{\epsilon_k})
	$$
	with  $a_{ij}^k$ $(\lambda, \Lambda)$-elliptic, $\delta_k = \text{o}(1)$, and 
	$$
		\max\{\epsilon_k,\inf\limits_{B_r(Y)} u_{\epsilon_k}\} =: \eta_k = \text{o}(1),
	$$
	but 
	\begin{equation}\label{lip 01-01}
		\sup\limits_{B_r/2(Y)} u_{\epsilon_k} \ge \theta_0 > 0,
	\end{equation}
	for some $\theta_0 > 0$ fixed. 
	Let $X_k$ be the point where $u_{\epsilon_k}$ attains its minimum in $B_r(Y)$ and denote $\sigma:= \text{dist}(B_r(Y), \partial \Omega) > 0$. Define the scaled function $v_k \colon B_{\sigma \epsilon_k^{-1}} \to \mathbb{R}$, by
	$$
		v_k(X) := \frac{ u_{\epsilon_k}(X_k + \epsilon_k X) }{\eta_k} 
	$$	
	One simply verifies that $v_k \ge 0 $ and it solves, in the distributional sense,
	\begin{equation}\label{KL eq001}
		\begin{array}{lll}
			\text{div} (a^k_{ij}(X) \nabla v_{k}) &=& \delta_k \cdot \left (\frac{\epsilon_k}{\eta_k} \beta_{1}(\frac{\eta_k}{\epsilon_k}v_{k}) \right )  \\
			&=& \text{o}(1),
		\end{array}
	\end{equation}
	as $k \to \infty$, in the $L^\infty$-topology.  Also, one easily checks that $v_k(0) \le  1$. Hence, by Harnack inequality, the sequence $v_k$ is uniform-in-$k$ locally bounded in $B_{\sigma \epsilon_k^{-1}}(0)$. From De Giorgi, Nash, Moser regularity theory, up to a subsequence, $v_k$ converges locally uniformly to an entire $v_\infty$. In addition, by standard Caccioppoli energy estimates, the sequence $v_k$ is locally bounded in $H^1$, uniform in $k$. Also by classical truncation arguments, up to a subsequence, $\nabla v_k(X) \to \nabla v_\infty(X)$ a.e. (see \cite{T0} and \cite{T1} for similar arguments). By ellipticity, passing to another subsequence, if necessary, $a_{ij}$ converges weakly in $L^2_\text{loc}$ to a  $(\lambda,\Lambda)$-elliptic matrix $b_{ij}$. Summarizing we have the following convergences:
	\begin{eqnarray}
		& v_k \to v_\infty   \text{ locally uniformly in } \mathbb{R}^n; \label{Conv1}\\
		& v_k \rightharpoonup v_\infty \text{ weakly in } H_\text{loc}^1(\mathbb{R}^n); \label{Conv2}\\
		 & \nabla v_k(X) \to \nabla v_\infty(X) \text{ almost everywhere in }  \mathbb{R}^n; \label{Conv3} \\
		& a_{ij}^k(X) \rightharpoonup b_{ij} \text{ weakly in } L_\text{loc}^2( \mathbb{R}^n). \label{Conv4}
	\end{eqnarray}
	Our next step is to the pass the limits above aiming to conclude that
	\begin{equation}\label{KL limit eq}
		\text{div} (b_{ij}(X) \nabla v_\infty) = 0, \quad \text{in } \mathbb{R}^n.
	\end{equation}
	This is a fairly routine procedure, but we will carry it out for the sake of the readers. Given a test function $\phi \in C^1_0(\mathbb{R}^n)$, let $k_0 \in \mathbb{N}$ be such that $B_{\sigma \epsilon_k^{-1}} \supset \text{Supp } \phi := K$. For $k> k_0$, we define the integrals
	$$
		\begin{array}{l}
			\mathscr{I}_k^1 := \displaystyle \int_{K} \langle a^k_{ij}(X) \nabla v_k, \nabla \phi \rangle dX;  \\
			\mathscr{I}_k^2 := \displaystyle \int_{K} \langle a^k_{ij}(X) \cdot( \nabla v_\infty - \nabla v_k), \nabla \phi \rangle dX;  \\
			\mathscr{I}_k^3 := \displaystyle \int_{K} \langle (b_{ij}-a_{ij}^k)(X) \cdot \nabla v_\infty, \nabla \phi \rangle dX;
		\end{array}
	$$
	and write
	\begin{equation}\label{KL passing limit2}
		\int_{\mathbb{R}^n} \langle b_{ij}(X) \nabla v_\infty, \nabla \phi \rangle dX = \mathscr{I}_k^1 + \mathscr{I}_k^2 + \mathscr{I}_k^3.
	\end{equation}	
	The purpose is to show that 
	$$
		\lim\limits_{k\to \infty}   \mathscr{I}_k^1 + \mathscr{I}_k^2 + \mathscr{I}_k^3 = 0.
	$$
	For that, let $\gamma>0$ be a given number small, positive number. It follows straight from \eqref{KL eq001} that $\mathscr{I}_k^1 = \text{o}(1)$ as $k \to \infty$. From \eqref{Conv4}, we also have straightly that  $\mathscr{I}_k^3 = \text{o}(1)$ as $k \to \infty$. Hence, for $k_1\ge k_0$, we have
	$$
		| \mathscr{I}_k^1| + | \mathscr{I}_k^3| \le \frac{\gamma}{2}.
	$$  
	Let us now analyze the convergence of $\mathscr{I}_k^2$. It follows from \eqref{Conv3} and Ergorov's theorem, that there exists a compact set $\tilde{K} \subset K$, such that
	$$
		\int_{K\setminus \tilde{K}} |\nabla \phi| dX \le \frac{\gamma}{5\Lambda \sup\limits_k \|\nabla v_k\|_2}
	$$ 
	and $k_2 \ge k_1$ such that 
	$$
		|\nabla v_k(X) - \nabla v_\infty(X)| \le \frac{3\gamma}{5\Lambda \|\nabla \phi\|_2\Leb({K})},
	$$
	in $\tilde{K}$, for all $k \ge k_2$. Hence, for $k \ge k_2$, we estimate, breaking it into two integrals on $\tilde{K}$ and on $K\setminus \tilde{K}$, and using H\"older inequality,  finally obtain
	$$
		| \mathscr{I}_k^2| \le  \dfrac{\gamma}{2}.
	$$
	We have henceforth proven the aimed convergence which gives \eqref{KL limit eq}.  Applying  Liouville theorem to $v_\infty$, we conclude that 
	$$
		v_\infty \equiv \text{Const.} < +\infty,
	$$
	for a bounded constant, in the whole space. The corresponding limiting function $u_\infty$ obtained from $u_{\epsilon_k}$ must therefore be identically zero. We now reach a contradiction with \eqref{lip 01-01} for $k \gg 1$. The Lemma is proven.
\end{proof}

	Before continuing, we remark that if $u_\epsilon$ is a solution to the original equation \eqref{e-Eq} and a positive number $\bar{\delta} > 0$ is given, then the zoomed-in function 
	$$
		\tilde{u}_\epsilon(X)=u_\epsilon(\sqrt{\bar{\delta}} X) 
	$$
	satisfies in the distributional sense the equation
	$$
		\div(\tilde{a}_{ij}(X) \nabla \tilde{u}_\epsilon)= \bar{\delta} \beta_\epsilon(\tilde{u}_\epsilon),
	$$	
	where $\tilde{a}_{ij}(X)  = {a}_{ij}(\sqrt{\bar{\delta}} X)$ is another $(\lambda, \Lambda)-$elliptic matrix. 

	\medskip
		
	We are in position to start delivering the proof of Theorem \ref{main}. Let $u_\epsilon$ be a bounded sequence of distributional solutions to \eqref{e-Eq} and $u_0$ a limit point in the uniform convergence topology. We assume, with no loss, that  $\xi = 0$, that is $u_0(0) = 0$. Within the statement of Lemma \ref{key lemma}, select 
	$$
		\theta = \dfrac{1}{2}.
	$$
	Since $u_\epsilon(0) \to 0$ as $\varepsilon \to 0$, Lemma \ref{key lemma} together with the above remark, gives the existence of a positive, universal number $\delta_\star>0$, such that if $0<\varepsilon \le \epsilon_0 \ll 1$, for $\tilde{u}_\epsilon(X) := u_\varepsilon(\sqrt{\delta_\star} X)$ we have
	$$
		\sup\limits_{B_{1/2}}  \tilde{u}_\epsilon(X) \le \frac{1}{2}.
	$$
	Passing to the limit as $\epsilon \to 0$, we obtain
	$$
		\sup\limits_{B_{\frac{\sqrt{\delta_\star}}{2} }} u_0 (X) \le \frac{1}{2}.
	$$
	Define, in the sequel, the rescaled function
	$$
		{v}^1(X) := 2 u_\varepsilon (\dfrac{\sqrt{\delta_\star}}{2} X). 
	$$
	It is simple to verify that ${v}^1$ satisfies 
	$$
		\text{div} ({a}^1_{ij} (X)\nabla {v}^1(X) ) = \delta_\star \beta_{2\epsilon}(v^1),
	$$
	in the distributional sense, where ${a}^1_{ij} (X) = a_{ij} (\sqrt{\delta_\star}X/2)$ is another $(\lambda, \Lambda)$-elliptic matrix. Once more, $v^1(0) \to 0$ as $\epsilon \to 0$, hence, for $\epsilon \le \epsilon_1 < \epsilon_0 \ll 1$, we can apply Lemma \ref{key lemma} to $v^1$ and deduce, after scaling the inequality back,
	$$
		\sup\limits_{B_{\frac{\sqrt{\delta_\star}}{4} }} u_0 (X) \le \frac{1}{4}.
	$$
	Continuing this process inductively, we conclude that for any $k\ge 1$, that holds
	\begin{equation}\label{d_thm}
		\sup\limits_{B_{\frac{\sqrt{\delta_\star}}{2^k} }} u_0 (X) \le \frac{1}{2^k}.
	\end{equation} 
	Finally, given $X\in B_{1/2}$ let $k \in \mathbb{N}$ be such that 
	$$
		\frac{\sqrt{\delta_\star}}{2^{k+1}}<  |X| \leq \frac{\sqrt{\delta_\star}}{2^k}.
	$$ 
	We estimate from \eqref{d_thm}
	$$
		\begin{array}{lll}
			u_0 (X) &\le& \displaystyle  \sup\limits_{B_{\frac{\sqrt{\delta_\star}}{2^k} }} u_0 (X) \\
			&\le& \displaystyle  \frac{1}{2^k} \\
			&\le& \displaystyle \frac{2}{\sqrt{\delta_\star}} |X|,
		\end{array} 
	$$
	and the proof of Theorem \ref{main} is concluded. \hfill $\square$
	
\medskip

Obviously, the (improved) regularity estimate granted by Theorem \ref{main} holds solely along the free boundary. For any point $Z \in \{u > 0 \}$, the best estimate available drops back to $C^{0,\alpha}$, for some unknown $0< \alpha$, strictly less than one. The question we would like to answer now is what is the minimum organization required on the medium so that solutions to the cavitation problem is locally Lipschitz continuous, up to the free boundary.

\begin{definition}
Given a large constant $K>0$, we say that a uniform elliptic matrix $a_{ij}(X)$ satisfies ($K$-Lip) property if for any $0<d<1$ and any  $h \in H^1(B_d)$ solving 
$$
	\div \left (a_{ij}(X) \nabla h \right ) = 0 \text { in } B_d
$$
in the distributional sense, there holds
$$	
	\|\nabla h\|_{L^\infty(B_{d/2})} \le \frac{K}{d} \cdot  \|h\|_{L^\infty(B_d)}.
$$ 
\end{definition}

It is classical that Dini continuity of the medium is enough to assure that $a_{ij}$ satisfies ($K$-Lip) property, for some $K>0$ that depends only upon dimension, ellipticity constants and the Dini-modulus of continuity of $a_{ij}$. Indeed under Dini continuity assumption on $a_{ij}$, distributional solutions are of class $C^1$.

Our next Corollary says that uniform limits of singularly perturbed equation \eqref{e-Eq} is Lipschitz continuous, up to the free boundary provided $a_{ij}$ satisfies  ($K$-Lip) property  for some $K>0$. The (by no means obvious) message being that when it comes to Lipschitz estimates, the homogeneous equation and the free boundary problem $\div (a_{ij}(X) \nabla u) \sim \delta_0(u)$ require the same amount of organization of the medium. 

\begin{corollary}\label{Cor-main} Under the assumptions of Theorem \ref{main}, assume further that $a_{ij}(X)$ satisfies ($K$-Lip) property for some $K$. Then, given a subdomain $\Omega' \Subset \Omega$, 
$$
	|\nabla u_0 (X)| \le C,
$$
for a constant that depends only  on dimension, ellipticity constants,  $\dist (\partial \Omega', \partial \Omega)$, $L^\infty$ bounds of the family and $K$.
\end{corollary}

\begin{proof} It follows from Theorem \ref{main} and property $K$ that $u_0$ is pointwise Lipschitz continuous, i.e., 
$$
	|\nabla u_0 (\xi)| \le C(\xi).
$$
We have to show that $C(\xi)$ remains bounded as $\xi$ goes to the free boundary. For that, let $\xi$ be a point near the free boundary $\partial \{u_0 > 0 \}$ and denote by $Y \in \partial \{u_0 > 0 \}$ a point such that
$$
	|Y-\xi | =: d = \dist (\xi, \partial \{u_0 > 0 \}).
$$
From Theorem \ref{main}, we can estimate
$$
	\sup\limits_{B_{d/2}(\xi)} u_0(\xi) \le \sup\limits_{B_{2d }(Y)} u_0(\xi) \le C\cdot 2d.
$$
Applying  ($K$-Lip) property to the ball $B_{d/2}(\xi)$, we obtain
$$
	|\nabla u_0 (\xi)| \le \frac{2K}{d} \cdot 2Cd = 4C\cdot K,
$$
and the proof is concluded.
\end{proof}


\section{Lipschitz estimates for the minimization problem} \label{Sct Lip AC}

Limiting functions $u_0$ obtained as $\epsilon$ goes to zero from a sequence $u_\epsilon$ of minimizers of functional \eqref{functional} are minima of the discontinuous functional \eqref{functional AC}. Hence, limiting minima are Lipschitz continuous along their free boundaries. Nonetheless, as previously advertised in Theorem \ref{Lip AC}, the sharp Lipschitz regularity estimate holds indeed for {\it any} minima of the functional \eqref{functional AC}, not necessarily for limiting functions. 

In this intermediate Section we shall comment on how one can deliver this estimate directly from the analysis employed in the proof of Theorem \ref{main}. In fact, the proof of Lipschitz estimate for minima of the functional \eqref{functional AC} is simpler than the proof delivered in previous section, which has been based solely on the singular equation satisfied. When a minimality property is available, the arguments can be rather simplified. For instance, strong minimum principle holds for local minima but is no longer available for a generic critical point. This is part of the reason why the arguments from previous section had be be based on blow-ups and Liouville theorem.

\begin{theorem}\label{main minimum}
Let $u_0\geq 0$ be a minimum to
	$$
		\mathscr{F} (u)=\int_{\Omega} \left \{\frac{1}{2}\langle a_{ij}(X)\nabla u,\nabla u\rangle + \chi_{\{u>0\}} \right \} dX
	$$ 
	and assume that $u_0(\xi)=0$. Then there exists a universal constant $C>0$, depending only on dimension, ellipticity constants,  $\dist (\xi, \partial \Omega)$ and its $L^\infty$ norm	such that
	$$
		u_0(X) \leq C|X -\xi|,
	$$
	for all point $X \in \Omega$.
\end{theorem}

The proof follows the lines designed in Section \ref{Sct Lip Funct}. We obtain the corresponding flatness Lemma as follows:

\begin{lemma}\label{key lemma min}
 	Fixed a ball $B_r(Y) \Subset  \Omega$ and given $\theta > 0$, there exists a $\delta > 0$, depending only on $B_r(Y)$, dimension, ellipticity constants and $L^\infty$ norm of $u_0$, such that if $u_0$ is a nonnegative minimum of
	$$
		\mathscr{F}^{\delta} (u)=\int_{\Omega} \left \{\frac{1}{2}\langle a_{ij}(X)\nabla u,\nabla u\rangle + \delta \cdot \chi_{\{u>0\}} \right \} dX,
	$$
	and $u_0(Y) = 0$, then
	$$
		\sup\limits_{B_\frac{r}{2}(Y)} u_0 \le \theta.
	$$
\end{lemma}

\begin{proof}
The proof follows by a similar tangential analysis of the proof of Lemma \ref{key lemma}, but in fact in a simpler fashion. The tangential functional, obtained as $\delta \to 0$, satisfies minimum principle, hence the limiting function, from the contradiction argument, must be identically zero.  

Here are some details:  suppose, for the sake of contradiction, that the Lemma fails to hold. It means, for a sequence $(\lambda, \Lambda)$-elliptic matrices, $a_{ij}^k$, and a sequence of minimizers $u_k$ of
	$$
		\mathscr{F}^{k} (u)=\int_{\Omega} \left \{\frac{1}{2}\langle a^k_{ij}(X)\nabla u,\nabla u\rangle + \delta_k \cdot \chi_{\{u>0\}} \right \} dX,
	$$
	with $\delta_k = \text{o}(1)$, and, say $\|u_k\|_\infty \le 1$,
	 \begin{equation}\label{lip 01-01 min}
		\sup\limits_{B_r/2(Y)} u_{k} \ge \theta_0 > 0,
	\end{equation}
	for some $\theta_0 > 0$ fixed. As in Lemma \ref{key lemma}, by compactness, up to a subsequence, $u_k \to u_0$. Passing the limits we conclude $u_0$ is a local minimum of
	$$
		\mathscr{F}^{\infty} (u)=\int  \frac{1}{2}\langle b_{ij}(X)\nabla u_0,\nabla u_0\rangle  dX.
	$$
	Since, $u_0 \ge 0$ and $u_0(Y) =0$, by the strong minimum principle, see for instance \cite[Theorem 7.12]{G},  $u_0 \equiv 0$. We now reach a contradiction with \eqref{lip 01-01 min} for $k \gg 1$. The Lemma is proven.
\end{proof}

Once we have obtained Lemma \ref{key lemma min}, the proof of Theorem \ref{main minimum} follows exactly as the final steps in the proof of Theorem \ref{main}.


\section{Gradient control in two-phase problems}

In this Section we show that Theorem \ref{main} as well as Theorem \ref{main minimum} hold for two-phase problems, provided a one-side control is {\it a priori} known. It is interesting to compare this with the program developed in \cite{C1, C2, C3}, where monotonicity formula yields similar conclusion.

Let us briefly comment on such generalization, in the (simpler) minimization problem. The singular perturbed one can be treated similarly. We start by placing the negative values of $u$ within a universally controlled slab, i.e.:
\begin{equation}\label{lower lim}
	\inf\limits_\Omega u  \geq -\delta_\star,
\end{equation}
for a universal value $\delta_\star > 0$. Such a condition is realistic for models involving very low temperatures, i.e., for physical problem near the absolute zero for thermodynamic temperature (zero Kelvin). A scaling of the problem places {\it any} solution into this setting. Within the proof of Lemma \ref{key lemma min}, one includes condition  \eqref{lower lim} in the compactness argument. Here is the two-phase version of Lemma  \ref{key lemma min}:
\begin{lemma}\label{key lemma min 2p}
 	Fixed a ball $B_r(Y) \Subset  \Omega$ and given $\theta > 0$, there exists a $\delta > 0$, depending only on $B_r(Y)$, dimension, ellipticity constants and $L^\infty$ norm of $u$, such that if $u$ is a changing sign minimum of
	$$
		\mathscr{F}^{\tilde{\delta}} (u)=\int_{\Omega} \left \{\frac{1}{2}\langle a_{ij}(X)\nabla u,\nabla u\rangle + \tilde{\delta} \cdot \chi_{\{u>0\}} \right \} dX,
	$$
	for $\tilde{\delta} \le \delta$, with 
	$$
		u_0(Y) = 0 \quad \text{ and } \quad \inf\limits_\Omega u  \geq -\delta,
	$$
	then
	$$
		\sup\limits_{B_\frac{r}{2}(Y)} |u| \le \theta.
	$$
\end{lemma}

The proof of Lemma \ref{key lemma min 2p} follows the lines of Lemma \ref{key lemma min}, noticing that, by letting $\delta  = \text{o}(1)$ in the compactness approach, the tangential configuration is too a nonnegative minima of a functional which satisfies minimum principle.

\begin{theorem} Let $u_0$ be a sign changing minimum of the functional
	$$
		\mathscr{F} (u)=\int_{\Omega} \left \{\frac{1}{2}\langle a_{ij}(X)\nabla u,\nabla u\rangle + \chi_{\{u>0\}} \right \} dX,
	$$ 
	with  $u_0(\xi)=0$, $-1\le u_0\le 1$. Assume $u^{-}$ is Lipschitz continuous at $0$. Then $u^{+}$ (and therefore $u$) is too Lipschitz at $0$ and 
	$$
		|\nabla u(0)| \le C |\nabla u^{-}(0)|.
	$$
\end{theorem}

\begin{proof}
We can assume, with no loss, $\Omega = B_2$ and $\xi = 0$. By universal continuity estimate, Theorem \ref{lim uniforme em C-alpha}, we can choose a universal number $0< \tau_0 \ll 1$, such that the function $ v \colon B_1 \to \mathbb{R}$, given by
$$
	{v}(X) := u(\tau_0 X),
$$
satisfies the hypothesis of Lemma \ref{key lemma min 2p}, for $\theta = \frac{1}{2}$. Selecting $0< \tau_0 \ll |\nabla u^{-}(0)|^{-1}$, even smaller if necessary, we can assure
$$
	|\nabla v^{-}(0)| \le \delta_{1/2},
$$ 
where $\delta_{1/2}$ is the number from Lemma \ref{key lemma min 2p} when we take $\theta = \frac{1}{2}$. Define in the sequel $v_2 \colon B_1 \to \mathbb{R}$ by
$$
	v_2(X) := 2 v(\frac{1}{2}X).
$$
Clearly $v_2$ is a minimum of a functional $\mathscr{F}^{\tilde{\delta}} $ for $\tilde{\delta} \le \delta_{1/2}$, $v_2(0) = 0$, and by Lemma \ref{key lemma min 2p}, it also verifies $-1\le v_2 \le 1$. We estimate
$$
	\inf\limits_{B_1} v_2 \ge - |\nabla v^{-}(0)|  \ge - \delta_{1/2}.
$$
Hence, $v_2$ is also within the hypothesis of Lemma \ref{key lemma min 2p}. Carrying the induction process shows that
$$
	\sup\limits_{B_{\tau_0 2^{-k}}} |u| \le 2^{-k}.
$$
Now, given $0< r \ll 1$, we choose $k \in \mathbb{N}$ such that $\tau_0 2^{-(k+1)} \le r \le \tau_0 2^{-k}$ and compute
$$
	\sup\limits_{B_r} |u| \le \sup\limits_{B_{\tau_0 2^{-k}}} |u| \le 2^{-k} \le \frac{2}{\tau_0} r.
$$
The Theorem is proven.
\end{proof}

\medskip

Similarly, one can use these set of ideas when a density control of the negative phase is given. For instance if 
\begin{equation}\label{density cond}
	\Leb (\{u<0\} \cap B_r) \le \delta_\star r^n,
\end{equation}
where $\delta_\star \ll 1$ is universally small, then Lipschitz regularity along the free boundary holds. Indeed, as before, one could add the constrain \eqref{density cond} within the compactness  approach, letting $\delta_\star = \text{o}(1)$, and the limiting configuration is too a nonnegative function. Now, within the induction procedure, condition \eqref{density cond} scales  properly, in the sense that at each scale, condition \eqref{density cond} holds with the same initial constant $ \delta_\star$. Compare for instance with \cite{LS}.

\section{Geometric estimates of the free boundary} 

In this Section we show how  the improved estimate given by Theorem \ref{main} (or else Theorem \ref{main minimum}) implies some geometric estimates  on the free boundary. Hereafter in this Section, $u_0 \ge 0$ will always denote a limit point obtained from a sequence of minimizers of the functional \eqref{functional}. We will denote by $\Omega_0$ the non coincidence set, $\Omega_0 := \{u_0 > 0 \} \cap \Omega$.  Unless otherwise stated, no continuity assumption is imposed upon the medium $a_{ij}$.


\begin{theorem}[Nondegeneracy]\label{strong non-degeneracy1}
 	Let  $\Omega'  \Subset \Omega$ be a given subdomain and $ Y \in  \Omega' \cap \overline{\{u_0 > 0\}}$ then 			
 	$$
 		\sup_{B_r(Y)}u_0 \ge c \cdot r.
 	$$
	for $r<\text{dist}(\Omega' , \partial \Omega)$.
\end{theorem}

\begin{proof}
	Letting $\epsilon \to 0$ in Theorem \ref{linear growth} we conclude $u_0$ grow linearly away from the free boundary. Owning Lipchitz regularity along $\partial \{u_0 > 0\} \cap \Omega'$, Theorem \ref{main}, we can then perform a polygonal type of argument {\it a la} Caffarelli, see for instance \cite[Lemma 4.2.7]{Teixeira book}, to establish such a strong non-degeneracy estimate.
\end{proof}

\begin{theorem}\label{Hausdorff}
	Given a subdomain $\Omega' \Subset \Omega$, there exists a constant $\theta > 0$, such that if $X_0\in \partial\Omega_0$ is  a free boundary point then	
	$$
		\mathcal{L}^n(\Omega_0\cap B_{r}(X_0))\geq \theta r^n,
	$$
	for all $0< r < \dist (\partial \Omega', \partial \Omega)$. Furthermore there is a universal constant $0< \varsigma \le 1$ such that
	$$
		\mathrm{dim}_{\mathcal{H}}(\partial\Omega_0)\leq n- \varsigma,
	$$
	where $dim_{\mathcal{H}}(E)$ means the Hausdorff dimension to the set $E$.
\end{theorem}
	
\begin{proof} It follows readily from non-degeneracy property, Theorem \ref{strong non-degeneracy1}, there exists a point $\xi_r \in \partial B_r(X_0)$ such that
$$ 
	u_0(\xi_r) \ge c r,
$$
for a constant $c>0$ depending only on the data of the problem. Now, for $0< \mu \ll 1$, small enough, there holds
\begin{equation}\label{mes01}
	B_{\mu r}(\xi_r) \subset \Omega_0.
\end{equation}
Indeed, one simply verifies that if
$$
	B_{\mu r}(\xi_r)  \cap \partial \{u_0 > 0 \} \not = \emptyset,
$$
then from Theorem \ref{main} we can estimate
$$
	c r \leq u_0(\xi_r) \leq \sup\limits_{B_{\mu r}(Z_0)} u_0 \le C   \mu r
$$
which is  a lower bound for $\mu$. Hence, if $\mu < c \cdot C^{-1}$, \eqref{mes01} must hold. Now, with such $\mu>0$ fixed, we estimate
$$
        \Leb\left (B_r(X_0) \cap \Omega_0 \right ) \ge
        \Leb \left (  B_r(X_0) \cap B_{\mu r}(\xi_r) \right ) \ge \theta r^n
$$
and the uniform positive density is proven. 

Let us turn our attention to the Hausdorff dimension estimate. Given $\sigma = X_0$ in $\partial \{u>0\}$, we choose
$$
	\sigma\prime = t \xi_r + (1-t)X_0,
$$
with $t$ close enough to $1$ as to 
$$
	B_{\frac{1}{2}\mu \cdot r }(\sigma\prime) \subset B_{\mu}(\xi_r) \cap B_r(\sigma) \subset B_r(\sigma) \setminus \partial \{u>0\}.
$$
We have verified $\partial \{u>0\} \cap B_{1/2}$ is $(\mu/2)$-porous, hence by a classical result, see for instance \cite[Theorem 2.1]{PR}, its Hausdorff dimension is at most $n - C\mu^n$, for a dimensional constant $C>0$.
\end{proof}

For problems modeled in a merely measurable medium, one should not expect an improved Hausdorff estimate for the free boundary. When diffusion is governed by the Laplace operator, then Alt-Caffarelli theory gives that $\varsigma = 1$. A natural question is what is the minimum organization of the medium as to obtain perimeter estimates of the free boundary. Next Theorem gives an answer to that issue.

\begin{theorem}\label{per}Assume $a_{ij}$ satisfy ($K$-Lip) property for some $K>0$. Then the free boundary has local finite perimeter. In particular $dim_{\mathcal{H}}(\partial\Omega_0) = n- 1$.
\end{theorem}

\begin{proof}
	Fixed a free boundary point $X_0 \in \partial \Omega_0$ and given a small, positive number $\mu$ one checks that
	\begin{equation}\label{strip lemma1}
		\int\limits_{\{0< u_0 < \mu \} \cap B_r(X_0)} |\nabla u_0|^2 \le C \mu r^{n-1}.
	\end{equation}
	This is obtained by integration by parts and Lipschitz estimate on $B_r$. In the sequel, we compare the left hand side of (\ref{strip lemma1}) with $\left | {\left\{0
<u_{0}< \mu\right\}\cap B_{r}(X_0)} \right |$. This is done by considering a finite overlapping converging, $\left\{B_j \right\}$, of $\partial\Omega_{0}$ by balls of radius
proportional to $\mu$  and centered on $\partial\Omega_{0} \cap
B_r(X_0)$.  In each ball $B_j$, we can
find subballs $B_{j}^{1},~B_{j}^{2}$ with the radii $\sim \mu$, such that
$$
	 u_{0}  \geq\frac{3}{4}\mu \text{ in }B_{j}^{1} \quad \text{and} \quad 
    u_0 \leq\frac{2}{3}\mu \text{ in } B_{j}^{2}.
$$
Existence of such balls is obtained by nondegeneracy property followed by Poincar\'e inequality.
Now, for $\mu \ll r$, we have
$$
    B_{r}(X_0)\cap\left\{0 < u_{0} <\mu \right\} \subset
    \bigcup 2B_{j}\subset B_{4r}(X_0).
$$
Finally, if we call $A := \left
\{ 0 <u_{0}<\mu\right\}$, the above gives
$$
    \begin{array}{lll}
        \displaystyle\int\limits_{B_{4r}(X_0)\cap A}|\nabla u_{0}|^2dX &\geq&
         \displaystyle\int\limits_{\big (\cup 2 B_{j}\big ) \cap A}|\nabla u_{0}|^2dX
         \\
         &\geq& \dfrac{1}{m}\displaystyle\sum
         \displaystyle\int\limits_{2B_{j}\cap A }|\nabla u_{0}|^2dX \\
         &\geq& c\displaystyle\sum \Leb(B_{j}) \\
         &\geq& c \Leb (B_{r}(X_0)\cap A ),
    \end{array}
$$
where $m$ is the total number of balls, which can be taken
universal, by Heine-Borel's Theorem. Combining the above estimate with \eqref{strip lemma1}, gives 
$$
	\Leb (\{0< u_0 < \mu   \} \cap B_{r}(X_0) ) \le C \mu r^{n-1},
$$
which implies the desired Hausdorff estimate by classical considerations. For further details, see for instance \cite[Chapter 4]{Teixeira book}.
\end{proof}

\bigskip

\noindent \textsc{Eduardo V. Teixeira} \hfill \textsc{Disson dos Prazeres} \\
\noindent Universidade Federal do Cear\'a  \hfill  Center for Mathematical Modeling \\
\noindent Departamento de Matem\'atica \hfill Universidad de Chile \\
\noindent Campus do Pici - Bloco 914 \hfill Beauchef 851, Edificio Norte - Piso 7\\
\noindent Fortaleza, CE - Brazil  \hfill Santiago, Chile \\
\noindent \texttt{teixeira@mat.ufc.br} \hfill  \texttt{dsoares@dim.uchile.cl}

\end{document}